\documentclass[11pt]{article}
\usepackage[margin=1in]{geometry}
\usepackage{amsmath,amssymb,amsthm,bbold}
\usepackage[numbers]{natbib}

\newcommand{\ep}{{\varepsilon}}

\newcommand{\e}{\varepsilon}
\newcommand{\diam}{\text{diam}}
\newcommand{\dist}{\text{dist}}

\usepackage{braket}

\newcommand{\abs}[1]{\left\vert #1 \right\vert}
\newcommand{\norm}[1]{\left\Vert #1 \right\Vert}
\newcommand{\parens}[1]{\left( #1 \right)}

\theoremstyle{plain}
\newtheorem{theorem}{Theorem}[section]
\newtheorem{corollary}[theorem]{Corollary}

\newtheorem{lemma}[theorem]{Lemma}

\theoremstyle{remark}
\newtheorem*{remark}{Remark}

\theoremstyle{definition}
\newtheorem{definition}{Definition}[section]

\usepackage{color}
\definecolor{red}{rgb}{.8,0,0}
\definecolor{green}{rgb}{0,.7,0}
\definecolor{blue}{rgb}{0,0,.8}

\title{An {A}lon-{B}oppana theorem for powered graphs and \\ generalized {R}amanujan graphs}
\author{Emmanuel Abbe and Peter Ralli\footnote{EPFL.  This research was funded in part by NSF CAREER Award CCF-1552131, ARO grant W911NF-14-1-0094, NSF grant DMS-1811735}}
\date{}
\begin{document}

\maketitle

\abstract{The $r$-th power of a graph modifies a graph by connecting every vertex pair within distance $r$. This paper gives a generalization of the {A}lon-{B}oppana Theorem for the $r$-th power of graphs, including irregular graphs. This leads to a generalized notion of Ramanujan graphs, those for which the powered graph has a spectral gap matching the derived {A}lon-{B}oppana bound. In particular, we show that certain graphs that are not good expanders due to local irregularities, such as Erd\H{o}s-R\'enyi random graphs, become almost Ramanujan once powered. A different generalization of Ramanujan graphs can also be obtained from the nonbacktracking operator. We next argue that the powering operator gives a more robust notion than the latter: Sparse Erd\H{o}s-R\'enyi random graphs with an adversary modifying a subgraph of $\log(n)^{\ep}$ vertices are still almost Ramanujan in the powered sense, but not in the nonbacktracking sense. As an application, this gives robust community testing for different block models.}

\section{Introduction}
%For a $d$-regular graph $G$, the Alon-Boppana Theorem shows that
%\begin{align}
%\lambda_2(G) \ge 2 \sqrt{d-1} - O\left(\frac{\sqrt{d-1}}{\mathrm{diam}(G)-4}\right),
%\end{align}
The Alon-Boppana Theorem implies that a family of $d$-regular $n$-graphs with adjacency matrix $A_n$ satisfies
\begin{align}
\lambda_2(A_n) \ge 2 \sqrt{d-1} - o_n(1).
\end{align}
A family of regular graphs is then called Ramanujan if it achieves this bound:
\begin{definition}
A family of $d$-regular $n$-graphs with adjacency matrix $A_n$ is {\it Ramanujan}, denoted here {\it A-Ramanujan}, if
\begin{align}
\lambda_2(A_n) \le 2 \sqrt{d-1}.
\end{align}
\end{definition}
Explicit constructions of such families were obtained in \cite{LPS,adam_1}, and it was shown in \cite{friedman} that random $d$-regular graphs $R_n$ are almost Ramanujan, in that
\begin{align}
\lambda_2(A_n) = 2 \sqrt{d-1} +  o_{n,\mathbb{P}}(1),
\end{align} 
where we use the notation $A_n=B_n+ o_{n,\mathbb{P}}(1)$ when $A_n-B_n$ tends to $0$ in probability when $n$ tends to infinity.
%where the $o_n(1)$ term can be positive \cite{}.
Obviously the above definitions are not directly relevant for irregular graphs. More specifically, Erd\H{o}s-R\'enyi (ER) random graphs with an expected degree $d$ will have their top two eigenvalues of order $\sqrt{\log(n)/\log\log(n)}$, due to eigenvectors localized on high-degree nodes, and therefore afford no spectral gap. Nonetheless, ER random graphs are similar to random $d$-regular graphs in various respects, e.g., their local neighborhoods for typical vertices are trees of either fixed or expected degree $d$. In particular, Lubotzky \cite{lubotzky_book} gives a definition of Ramanujan that generalizes to irregular graphs, where $G$ is Ramanujan if for every non-trivial eigenvalue $\lambda$ of $A(G)$, $\abs{\lambda}\leq \rho(\hat{G})$ where $\rho(\hat{G})$ is the spectral radius of the universal cover $\hat{G}$ of $G$. We observe that this definition does not fix the previously mentioned issue, as the spectral gap of $G$ may be in some sense maximally large given the universal cover, but this property does not overcome the problem of the corresponding eigenvectors simply isolating on high-degree nodes.
Thus one may wonder whether ER random graphs could also be good expanders, or even almost Ramanujan, if their local irregularity could be smoothed out. We will next discuss how to formalize and quantify such statements, and give motivating applications. \\ 

\noindent
{\bf Generalized Ramanujan: beyond the adjacency operator.} We start with a concrete example of a generalization of the Ramanujan property that can be obtained using the nonbacktracking operator of the graph. Given a graph $G$, the nonbacktracking matrix $B_G$ is defined by the matrix on the set of directed edges of the graph (i.e., its dimension is twice the number of edges), and for two directed edges $e=(e_1,e_2)$, $f=(f_1,f_2)$, $B_{e,f}=\mathbb{1}(e_2=f_1)\mathbb{1}(e_1\ne f_2)$. It was shown in \cite{terras} that for regular graphs, the Ramanujan property can be equivalently defined using the nonbacktracking spectral gap:
\begin{definition}
A family of $d$-regular $n$-graphs with nonbacktracking matrix $B_n$ is {\it B-Ramanujan}, if
\begin{align}
|\lambda_2(B_G)| \le \sqrt{d}, \label{Bram}
\end{align}
where a graph statisfying \eqref{Bram} is also said to satisfy the {\it graph Riemann hypothesis \cite{terras},} since the eigenvalues of the nonbacktracking operator are the reciprocal of the poles of the Ihara zeta function of the graph \cite{hashimoto, terras}.
\end{definition}
This definition is indeed equivalent to the former definition for regular gaphs.
\begin{lemma}\cite{terras}
For regular graphs, A-Ramanujan is equivalent to B-Ramanujan.
\end{lemma}
Further, it was shown in \cite{bordenave} that the B-Ramanujan definition extends more naturally to some irregular graphs than the A-Ramanujan definition, with the ER random graph being almost B-Ramanujan.
\begin{theorem}\cite{bordenave}
For a random ER graph, 
\begin{align}
|\lambda_2(B_n)| = \sqrt{\lambda_1(B_n)} + o_{n,\mathbb{P}}(1),
\end{align}
with $\sqrt{\lambda_1(B_n)} =  d + o_{n,\mathbb{P}}(1)$ in this case.
\end{theorem}
Therefore, the nonbacktracking operator meets our objective to turn ER random graphs into almost Ramanujan graphs {\it in the nonbacktracking domain}. We next argue that this approach can be further improved.\\ 

\noindent
{\bf Symmetry and robustness.} The B-Ramanujan definition suffers from two drawbacks: (1) Complex spectrum: as opposed to $A$, the matrix $B$ is no longer symmetrical and thus has a complex spectrum. This makes some of the spectral  intuition more delicate, where expansions are in terms of directed walks (that do not backtrack) and where the Courant-Fisher theorem (connecting cuts to eigenvalues) requires the use of oriented path symmetry. In particular, a tight Alon-Boppana theorem in the nonbacktracking domain is not obtained in \cite{bordenave}. (2) Robustness: the nonbacktracking operator meets the objective of making ER random graphs almost Ramanujan, but this property is lost once one slightly deviates from such models. For instance, perturbing the ER graph by adding a clique of size $c=\Omega(\sqrt{d})$ edges already makes the perturbed graph far from $B$-Ramanujan.

Some solutions have been proposed for these issues. First, the Bethe-Hessian operator \cite{florent_bethe} has been shown to essentially act as a symmetrized version of the nonbacktracking operator, however it does not fix the robustness issue. In \cite{colin3cpam}, the generalized notion of $r$-nonbacktracking operator is used to gain generality in the proofs, but this is still nonsymmetrical and the complexity of the eigenpair computation increases significantly with $r$.

In \cite{powering_arxiv,ABRS19}, graph powering was proposed to address issues (1) and (2), testing robustness on a geometric block model, with parallel results in \cite{stephan2018robustness,stephan_colt} using a related operator based on graph distances. However, these papers no longer investigate the connection to Ramanujan graphs, which is explicit in the case of the nonbacktracking operator \cite{bordenave} (cf.\ previous paragraphs). \\

\noindent
{\bf This paper.} In this paper we consider the symmetric operator of graph powering, and investigate its robustness and extremal spectral gap properties. 
The $r$-th graph power $G^{(r)}$ of a graph $G$ modifies the graph by adding edges between any vertex pair at distance less or equal to $r$ \cite{ABRS19}. Equivalently, the adjacency matrix of $G^{(r)}$ is given by $A^{(r)}=\mathbb{1}((I+A)^r \ge 1)$, where the indicator function is applied point-wise to the matrix. We are typically interested in $r$ large but significantly less than the graph diameter (otherwise powering turns the graph into a complete graph). In general, a regular graph may no longer be regular once powered, so even for regular graphs, we cannot bound the spectral gap for powered graph simply by using the Alon-Boppana result with degree $d^r$. Nonetheless, if we take a regular graph of girth larger than $2r$, then the $r$-th power is regular and the Alon-Boppana Theorem gives the bound $\lambda_2 \ge 2\sqrt{d^r-1}$, so approximately $2 d^{r/2}$ for large $r$ or $d$. In fact, we shall see that random d-regular graphs have a second eigenvalue around $r d^{r/2}$ instead of $2 d^{r/2}$ for large $r$ \cite{ABRS19}.  For this reason one might conclude that in the powered domain, random $d$-regular graphs are not almost Ramanujan, creating contrast to the classical definition, and suggesting that powering may be misleading for generalizing Ramanujan graphs. The main result of this paper shows that this argument is false.  Instead we will observe that applying the general Alon-Boppana bound to powered graphs is suboptimal, since powered graphs are not arbitrary graphs - instead, they are {\it powers} of arbitrary graphs. 

We show an Alon-Boppana bound that applies to powers of (possibly irregular) graphs and which matches the scaling $r d^{r/2}$ for random $d$-regular graphs. Further, it is shown that both ER and random regular graphs have a comparably large and `optimal' spectral gap in the powered domain, i.e.,  they are almost P-Ramanujan, just as for B-Ramanujan. However, we show that this P-Ramanujan definition is more robust to local density variations (e.g., degrees) and adversaries than the B-Ramanujan definition: an adversary modifying a subgraph containing $\log(n)^{\ep}$ vertices in ER or RR cannot disrupt the P-Ramanujan property, while the B-Ramanujan property is lost after such a perturbation. 
We finish the introduction by motivating why such robust extensions are useful for spectral algorithms. \\ 

\noindent {\bf Detection from the spectrum.} Consider the problem of detecting the presence of a hidden structure in a graph, such as communities in the stochastic block model. This means that we want to distinguish between two cases, either the graph is drawn from ER$(n,d/n)$, or, on the other hand, it is the assembly of two independent ER$(n/2,a/n)$ subgraphs with a random bipartite graph connecting each pair of vertices across the groups independently with probability $b/n$. Denoting by $A$ the adjacency matrix of the graph, we have   
\begin{align}
H_0: A \sim \mu_0 = ER(n,d/n) \qquad H_1: A \sim \mu_1 = SBM(n,a/n,b/n) ,
\end{align} 
and we want to identify parameter regimes for which we can distinguish the two models with probability $1/2 + \Omega_n(1)$, or equivalently, for which $TV(\mu_0,\mu_1)=\Omega_n(1)$.
One can view the SBM adjacency matrix as an ER matrix $A$ perturbed as $A+Z$ where $Z$ adds/subtracts edges within/across clusters with the specified probabilities. If $(a+b)/2 \ne d$,  the average degree or edge density of the graph allows to distinguish the two models, so we consider the case where $(a+b)/2 = d$.
If we can see enough i.i.d.\ realizations of the graph (with vertex labels), the law of large numbers allows to approximate the expected adjacency matrix which is distinct in the two models: both have min-eigenvalues $0$ and max-eigenvalues $d$, but the SBM has an additional eigenvalue at $(a-b)/2$, which allows to distinguish the two models.  However, we are interested in the case where we see a single sample of the model. In this case, rather than a $0$ eigenvalue of large multiplicity, we have a distribution that spills over the informative eigenvalues, with high-degree nodes forming the largest eigenvalues in both models. So the spectrum of $A$ does not \textit{a priori} allow us to distinguish the two models, i.e., to decide correctly with probability $1/2 + \Omega_n(1)$ on the hypothesis.  

Rather than using the spectrum of $A$, we can use the cycle counts, as originally proposed in \cite{Mossel_SBM1}. This allows us to distinguish the models down to the optimal Kesten-Stigum (KS) threshold, i.e., $\lambda_2(SBM) > \sqrt{\lambda_1(SBM)}$, which reads $(a-b)/2 > \sqrt{(a+b)/2}$. One can also use spectral methods, not based on the adjacency matrix but on the nonbacktracking matrix \cite{bordenave}, which does not suffer from ER irregularities due to the weak Ramanujan property: its second eigenvalue is $\sqrt{(a-b)/2}+ o_{n,\mathbb{P}}(1)$ in the ER case, and $(a-b)/2+ o_{n,\mathbb{P}}(1)$ in the SBM case due to the community eigenvector. Thus the second eigenvalue of the nonbacktracking matrix allows us to solve the distinguishing problem down to the optimal KS threshold. 

The relevance of the almost-Ramanujan property in the B-domain is now clear: a large spectral gap for the null model (ER) leaves more room for the community signal (i.e., $\lambda_2$) to be visible in the SBM, and thus gives a broader range of parameters for which testing is solvable. 

Consider now the {\bf robust testing problem}, i.e., the same problem as above with the addition of an adversary that can modify adversarially edges incident to a fixed number $c$ of vertices before one sees the graph. It is not hard to check that a budget of $c = \Omega(a+b)$ suffices to disrupt the two previous methods based on cycle counts and the nonbacktracking operator. 
However, for graph powering and for $c = o\parens{\tfrac{\parens{(a-b)/2}^r}{\log(n)\sqrt{(a+b)/2}^{\,r-1}}}$, we will prove that that the ER model perturbed by such an adversary affords still a maximal spectral gap in the powered domain. This allows one to distinguish the models down to the KS threshold despite such adversaries --- See Corollary \ref{corr_ER_SBM}. For this  case, a similar result has recently been obtained in parallel work \cite{stephan2018robustness} for the SBM using a slightly different operator based on the distance matrix of the graph. \cite{stephan2018robustness} further covers the case of weak recovery. Here we instead add the case of the regular-SBM.  

\section{Results}
\subsection{Notations}

We will start by recalling some standard notations.  In a graph $G$, $\dist_G(v,w)$ is the graph distance metric, measuring the length (in edges) of the shortest $v-w$ walk in $G$.  If $G$ is a finite connected graph, $\diam(G)$ is the maximum graph distance between any pair of vertices.  If $G$ has $|V| = n$ the adjacency matrix $A(G)$ is an $n\times n$ matrix indexed by $V$ in which $A_{ij} = 1$ if $i\sim_G j$ and $0$ otherwise.  The eigenvalues of $A$ are $\lambda_1\geq\lambda_2\geq \lambda_3\dots$.

Let $G$ be a graph, $G$ may have self-loops but we do not allow repeated edges.  If $r\geq 1$, the $r$-th power of $G$ is $G^{(r)}$, the graph with vertex set $V(G)$ and an edge $\{x,y\}$ iff $\dist_G(x,y)\leq r$.  This definition was introduced in \cite{ABRS19}.  $A^{(r)}$ is the adjacency operator of $G^{(r)}$.  The graph power $G^{(r)}$ will by definition contain a self-loop at every vertex, it does not contain repeated edges. 

In order to model community behavior, we sample random graphs from the (balanced $2$-community) Stochastic Block Model:  a graph $G$ sampled from $SBM(n,a/n,b/n)$ is a random graph on $n$ vertices generated by the following two steps.  First, each vertex is put in community $X_1$ or $X_2$ uniformly and independently.  Second, for every pair of vertices $v,w$, $\{v,w\}$ is taken to be an edge with probability $a/n$ if $v$ and $w$ are in the same community and $b/n$ if not.  The $SBM$ is a generalization of the well-known Erd\H{o}s-R\'enyi graph $ER(n,d/n) := SBM(n,d/n,d/n)$.

We want to understand recovery of the $SBM$ communities: that is, given a graph pulled from $SBM(n,a/n,b/n)$, how well can we determine which vertices are in which community.  Because of symmetry we can never determine which community is $X_1$ as opposed to $X_2$, but the aim is to guess a community more accurately than a random guess would allow.  More formally, we solve the problem of \textit{weak recovery} if we can, with high probability, match more than $\tfrac{1}{2}+\ep$ of vertices to the correct community (up to a relabelling of the communities), where $\ep$ is some positive constant.  The \textit{Kesten-Stigum (KS) threshhold} $\tfrac{a-b}{2} = \sqrt{\tfrac{a+b}{2}}$ is a limit on weak recovery: when $\tfrac{a-b}{2} \leq \sqrt{\tfrac{a+b}{2}}$ weak recovery is not possible\cite{Mossel_SBM1}.

The Random $d$-regular graph on $n$ vertices is $RR(n,d)$.  The Regular Stochastic Block Model $RSBM(n,a,b)$ (as seen in \cite{Brito16}) is a random graph generated by first choosing uniformly at random a partition of the $n$ vertices into two equal communities, then putting a random $a$-regular graph on each community and a random $b$-regular bipartite graph between the communities.

\subsection{Alon-Boppana for powered graphs}

We investigate the maximum size of the spectral gap following graph powering.  We modify well-known methods of finding a lower bound on the second-largest eigenvalue in a graph in order to derive a version of the Alon-Boppana theorem for graph powering.

Recall that the Alon-Boppana result for (non-powered) $d$-regular graphs is  $$\lambda_2(A) \ge (1-o_{\diam(G)}(1)) 2 \sqrt{d-1}.$$  A Ramanujan graph is one for which the lower bound is tight, as first investigated in \cite{LPS}.

Friedman \cite{friedman} argued that $d$-regular random graphs are almost Ramanujan with high probability.  We will replicate this result under powering, arguing that with high probability, a $d$-regular random graph under powering is almost $r$-Ramanujan.

The Alon-Boppana like bound for powered graphs is:

\begin{theorem}\label{abp}
Let $\{G_n\}_{n \ge 1}$ be a sequence of graphs such that $\diam(G_n)=\omega(1)$, and $\{r_n\}_{n \ge 1}$ a sequence of positive integers such that $r_n=\e \cdot \diam(G_n)$. Then, 
\begin{align}
\lambda_2(G_n^{(r_n)}) \ge (1-o_\e(1)) (r_n+1) \hat{d}_{r_n}^{\, r_n/2}(G_n),
\end{align}
where
\begin{align}
 \hat{d}_r(G) &= \left( \frac{1}{r+1} \sum_{i=0}^r \sqrt{\delta^{(i)}(G)\delta^{(r-i)}(G)} \right)^{2/r},\\
 \delta^{(i)}(G) &= \min_{(x,y) \in E(G)} |\{ v \in V(G): \dist_G(x,v)=i,\dist_G(y,v)\ge i \}|.
\end{align}
\end{theorem}

In \cite{powering_arxiv,ABRS19}, the authors (jointly with Boix and Sandon) prove that the quantity $\hat{d}$ is $d \pm o(d)$ with high probability in a $d$-regular random graph, for such graphs we get the following version of Theorem \ref{abp} 

\begin{theorem}\label{lemma_reg}
Let $G$ be a random $d$-regular graph and $r = \e \log(n)$, where $\e \log d <1/4$. Then, with high probability, 
\begin{align}\lambda_2(G^{(r)}) \geq (1-o(1))(r+1)\sqrt{d}^{\, r}.\end{align}
\end{theorem}

We say that a graph $G$ is $r$-Ramanujan if the bound of Theorem \ref{abp} is tight, that is, if $\lambda_2(G^{(r)})\leq (1+o(1))(r+1)\hat{d}^{r/2}$.  We demonstrate that a class of Ramanujan graphs are also $r$-Ramanujan, and therefore that $r$-Ramanujan graphs exist.

\begin{lemma}\label{lemma_girth}
Let $G$ be a $d$-regular Ramanujan graph with girth $g$, and let $2r < \mathrm{girth}(G)$.  Then $G^{(r)}$ is $r$-Ramanujan, so that $\lambda_2(G^{(r)})= \parens{1+o_d(1)}(r+1) d^{r/2}$.  
\end{lemma}

As first seen in the original construction of \cite{LPS}, there are families of Ramanujan graphs may have girth $\Theta(\log(n)/\log(d)$; it is straightforward to observe that this is an upper bound on girth.  Using such graphs, we can now construct $r$-Ramanujan powered graphs.

\subsection{Robustness of graph powering}

If a graph contains large cliques but otherwise appears to be randomly generated (such as by the stochastic block model), analysis of the leading eigenvectors will reliably identify those cliques rather than any communities that may exist.  We will investigate what happens to such graphs under powering.  Observe that a shortest path of length $r$ may have, at most, $1$ edge from any clique.  Intuitively, powering the graph means gives a graph whose edges are paths in $G$ that take edges mostly from the "expander part" of the graph with at most one edge from any large clique, and so we can observe the community behaviour.  In fact the results will be more general; we will consider any perturbation of an expander graph and not just the addition of a clique.

As we wish to study sparse random graphs which are locally tree-like, we find it useful to fist examine the properties of the tree.  In particular we will examine a graph that is a $d$-regular infinite tree with one $c$-clique.

\begin{definition}
Let $c,d$ be positive integers.  $T_{d,c}$ is the graph consisting of $c$ copies of the $d$-regular infinite tree $T_d$ that are attached by a $c$-clique containing exactly one vertex from each.
\end{definition}

Recall that the spectral radius of a graph $G = (V,E)$ with adjacency operator $A$ is $$\rho(G) = \sup_f \frac{\langle Af,f\rangle}{\langle f,f\rangle},$$ taken over all functions $f$ with $0 < \langle f,f\rangle < \infty$.

\begin{theorem}
$\rho(T_{d,c}) =\Theta (c + \sqrt{d} )$. 
\end{theorem}

\begin{proof}

Suppose $E$ is the disjoint union of edge sets $E_1$ and $E_2$.  Denote $G_i = (V,E_i)$ with adjacency matrix $A_i$, so that we can express $A = A_1 + A_2$.  It is straightforward to see that $\rho(G_i) \leq \rho(G) \leq \rho(G_1)+\rho(G_2)$, and thus $\rho(G) \approx \rho(G_1) + \rho(G_2)$.

For $T_{d,c}$, we will take $E_1$ to be the set of all edges within the $c$-clique and $E_2$ to be all other edges.  $G_2$ is a union of $c$ different $d$-regular trees.  It is well-known that $\rho(G_1) = c-1$ and $\rho(G_2) = 2\sqrt{d}$, the result follows.
\end{proof}

It is easy to see here that $c \approx \sqrt{d}$ is a threshold for $\rho(T_{d,c})$; if $c \gg \sqrt{d}$ then the size of the clique controls $\rho(T_{d,c})$, if $c \ll \sqrt{d}$ then the $\sqrt{d}$ term controls $\rho$ and the effect of the clique is hidden within the spectrum of $T_{d,c}$.

First, we will analyze the tree with added clique $T_{d,c}$.  When $c \gg \sqrt{d}$ then $\rho_2(T_{d,c})\approx c$.  This can be seen by counting the number of closed $k$-walks terminating at some vertex $x$ in the $K_c$; this number is $\approx (c-1)^k$ counting only those walks that stay within the $K_c$.  A similar strategy for obtaining a large number of $k$-walks in $T_{d,c}^{(r)}$ fails in that each step of such a walk represents an $\leq r$-path, that path can contain at most one edge of the $K_c$ and so it will either be short or travel into the "tree portion" of $T_{d,c}$.  In this way we "smooth over" the clique, obtaining the following result.

\begin{theorem} \label{thm_tr_cliques}
Let $G = T_{d,c}^{(r)}$.  Then $$\max\biggl\{(1-o(1)(r+1)\sqrt{d},c\biggr\}\sqrt{d}^{\,r-1}\leq \rho(G) \leq (1+o(1))\parens{c+(r+1)\sqrt{d}}\sqrt{d}^{\,r-1}.$$
\end{theorem}

Having established these results for trees with a cluster added, we can apply the methods to the study of more general expander graphs.  Before dealing with the specific examples of random graphs, we will give some general results on how adding edges to a graph impacts the eigenvalues of $G^{(r)}$.

\begin{definition}
Let $G$ be a graph on vertex set $[n]$ and let $H$ be a graph whose vertex set is a subset of $[n]$.  Then $G+H$ is the graph with vertex set $[n]$ and satisfying the equation $E(G+H)\Delta E(G) = E(H)$, in other words, $G+H$ is obtained from $G$ by adding or removing all edges of $H$ as applicable.
\end{definition}

We will use the following simple theorem for bounding eigenvalues of graphs of the form $G+H$.

\begin{theorem} \label{thm_plus_graph_bound}
Let $k \geq 1$.  Then $\abs{\lambda_k(A_{G+H})-\lambda_k(A_G)}\leq \norm{H}$.
\end{theorem}

\begin{proof}
Weyl's inequality states that $\abs{\lambda_k(A_{G+H})-\lambda_k(A_G)}\leq \norm{A_{G+H}-A_G}$.  Observe that entrywise $[A_H]_{ij} = \abs{[A_{G+H}-A_G]_{ij}}$, and so $\norm{A_{G+H}-A_G}\leq A_H$
\end{proof}

We can also prove a version of Theorem \ref{thm_plus_graph_bound} for graphs of the form $\parens{G+H}^{(r)}.$  This is the key result which we will later apply to various models of random graphs.

\begin{theorem}\label{thm_gen_clique_bound}
Let $G$ be a graph with $c < |V(G)|$, and let $H$ be a graph whose vertex set consists of at most $c$ elements of $V(G)$.  Define $D^{(i)}(G+H)$ to be the maximum degree in $G^{(i)}$ over all the vertices of $H$.  Then $$\abs{\lambda_k((G+H)^{(r)}) - \lambda_k(G^{(r)})} \leq \sum_{q = 0}^{r-1} c\max_i \sqrt{D^{(i)}D^{(q-i)}}.$$
\end{theorem}

\begin{proof}
Let $A = A((G+H)^{(r)})$, $A_1 = A(G^{(r)})$, $A_2 = A-A_1$ and let $\overline{A}_2$ be the entrywise absolute value of $A_2$.    Weyl's inequality tells us that $\abs{\lambda_k(A) - \lambda_k(A_1)}\leq \norm{A_2}\leq \norm{\overline{A}_2}$.  The remainder of the proof is bounding $\norm{\overline{A}_2}$.  We will proceed similarly to the method of Theorem \ref{thm_tr_cliques}.

Observe that $\overline{A}_2$ is the adjacency matrix for a graph $G_2$.  If $x\sim_{G_2} y$, it must be the case that there are vertices $v,w\in V(H)$ so that $\dist_G(x,v) + \dist_G(y,w) \leq r-1$, i.e., there is a path of length at most $r$ between $x$ and $y$, and that path passes through $H$. Further partition $E(G_2)$ into $E_q:0\leq q\leq r-1$, where $\{x,y\}\in E_q$ if $\{x,y\}\in E(G_2)$ and $\dist_G(x,H) + \dist_G(y,H) = q.$  Let $B_q$ be the adjacency matrix of $E_q$, clearly $\norm{\overline{A}_2}\leq \sum_q \norm{B_q}$.

Now we will bound $\norm{B_q}$.  Suppose $\dist_G(x,H) = i$, then $x$ has at most $c\cdot D^{(q-i)}$ $E_q$-neighbors $y$, and no such neighbors if $i > q$.

Define $f_q$ so that $f_q(x) = \sqrt{D^{(q-i)}}$ if $\dist_G(x,H) = i$.  For such $x$,  $$\frac{B_qf_q(x)}{f_q(x)} \leq \frac{c\cdot D^{(q-i)}\sqrt{D^{(i)}}}{\sqrt{D^{(q-i)}}}= c\sqrt{D^{(i)}D^{(q-i)}}.$$  It follows that $\norm{B_q}\leq c\max_i \sqrt{D^{(i)}D^{(q-1)}}$, this completes the proof.
\end{proof}

\subsubsection{$SBM$ with perturbation}

We apply Theorem \ref{thm_gen_clique_bound} to the case of an  Erd\"{o}s-R\'{e}nyi random graph or a sparse SBM.

\begin{theorem} \label{thm_ER_SBM}
Let $G = SBM(n,a/n,b/n)+H$ where $|V(H)|\leq c$.  There is a universal constant $\alpha$ so that if $\sqrt{(a+b)/2} \leq (a-b)/2$ (the KS threshhold), then, independently of the choice of $H$, with high probability, $$(1-o(1))\parens{\tfrac{a-b}{2}}^r-c\log(n)^{\alpha}\sqrt{\tfrac{a+b}{2}}^{\,r-1} \leq  \lambda_2(G^{(r)})\leq c\log(n)^{\alpha}\sqrt{\tfrac{a+b}{2}}^{\,r-1}+(1+o(1))\parens{\tfrac{a-b}{2}}^r.$$

On the other hand, if $\sqrt{(a+b)/2} \leq (a-b)/2$, then with high probability
$$\parens{\sqrt{\tfrac{a+b}{2}}-c}\log(n)^{\alpha}\sqrt{\tfrac{a+b}{2}}^{\,r-1}\leq \lambda_2(G^{(r)})\leq \parens{c+\sqrt{\tfrac{a+b}{2}}}\log(n)^{\alpha}\sqrt{\tfrac{a+b}{2}}^{\,r-1}.$$

In particular, if $G = ER(n,d/n)+H$ (i.e., the $SBM$ with values $a = b = d$), then WHP $$\parens{\sqrt{d}-c}\log(n)^{\alpha}\sqrt{d}^{\,r-1}\leq \lambda_2(G^{(r)})\leq \parens{c+\sqrt{d}}\log(n)^{\alpha}\sqrt{d}^{\,r-1}.$$
\end{theorem}

\begin{remark}
Note that we are thinking of $r = \log(n)^{\gamma}$ where $\gamma$ is a constant.  If $c = \log(n)^{\ep}$, then in this result, we obtain the bound $\lambda_2(ER(n,d/n)^{(r)})\leq \log(n)^{\alpha}d^{r/2}$ for the original $ER$ graph and $\lambda_2(G^{(r)})\leq \log(n)^{\alpha+\ep}d^{r/2}$ for the perturbed graph.  Our Alon-Boppana result for powering states that (if the unpowered graph is $d$-regular) then $\lambda_2\geq (r+1)d^{r/2} = \log(n)^{\gamma}d^{r/2}$.  Because the upper bounds for $\lambda_2(ER(n,d/n)^{(r)})$ and $\lambda_2(G^{(r)})$ are tight up to a power of $\log(n)$ we say that those graphs are both almost Ramanujan.
\end{remark}

\begin{proof}
This is an application of Theorem \ref{thm_gen_clique_bound}.

Consider the case that $G = SBM(n,a/n,b/n)+H$, define $G' = SBM(n,a/n,b/n)$ so that $G = G'+H$.  By Lemma 4.6 of \cite{ABRS19}, in the high-probability case $D^{(i)} \leq \log(n)((a+b)/2)^i$, and so, because $\lambda_2(G'^{(r)}) = \max\parens{\log(n)^{\alpha}\sqrt{\tfrac{a+b}{2}}^{r},\parens{\tfrac{a-b}{2}}^r}$ by Theorem 2.6 of \cite{ABRS19}, the results are straightforward.
\end{proof}

We will use the following result to solve the distinguishability problem.  That is, suppose with equal probability either the $ER(n,\tfrac{a+b}{2n})$ or $SBM(n,a/n,b/n)$ random graph model is chosen, and then a graph $G$ is drawn from that model at random and perturbed to $G+H$ by an adversarial choice of $H$ with the constraint $|V(H)|\leq c$.  Then for what values of $a,b,c$ is it possible to guess with high probability which of the two models $G$ comes from?

\begin{corollary} \label{corr_ER_SBM}

Assume $c = o\parens{\tfrac{\parens{(a-b)/2}^r}{\log(n)\sqrt{(a+b)/2}^{\,r-1}}}.$  

Let $G = SBM(n,a/n,b/n)+H$.  Independently of the choice of $H$, with high probability $$\lambda_2(G^{(r)}) = (1\pm o(1)) \parens{\tfrac{a-b}{2}}^r.$$

Let $G = ER(n,\tfrac{a+b}{2}/n)+H$.  Independently of the choice of $G$, with high probability $$\lambda_2(G^{(r)}) = o( \parens{\tfrac{a-b}{2}}^r).$$  
\end{corollary}

The proof of each statement is just an application of Theorem \ref{thm_ER_SBM}.

\begin{remark}  

A common method of solving the distinguishability problem is by examining the number of $m$-cycles \cite{colin3}. In brief, the number of cycles is $~\tfrac{1}{2m}\parens{d^m \pm d^{m/2}}$ in an $ER$ graph and $~\tfrac{1}{2m}\parens{d^m + (\tfrac{a-b}{2})^m \pm d^{m/2}}$ in an $SBM$, so that the decision is possible up to the $KS$ threshhold. However this method is not robust to adversarial perturbation of the graph.  If the perturbation is a $c$-regular clique where $c>>d$, the number of $m$-cycles is $~\tfrac{1}{2m}\parens{c^m + d^2c^{m-2} \pm \sqrt{d^2c^{m-2}}}$ for both the $ER$ and $SBM$ random graph models, this makes the decision impossible.  But the method of graph powering lets us solve this decision problem even with the addition of much larger cliques.  This result is similar to one found in the parallel work of Stephan and Massouli\'e \cite{stephan2018robustness}, working with the distance matrix rather than $A^{(r)}$.
\end{remark}

\begin{remark}
Implicit in the result of Theorem \ref{thm_ER_SBM} is that if $G = SBM(n,a/n,b/n)+H$ under the hypotheses of Corrolary \ref{corr_ER_SBM}, then the second eigenvector $v_2$ of $G^{(r)}$ will approximate the second eigenvector of $SBM(n,a/n,b/n)^{(r)}$.  Theorem 2.6 of \cite{ABRS19} tells us that the second eigenvector of $SBM(n,a/n,b/n)^{(r)}$ is useful for weak recovery of the communities.  
\end{remark}

\subsubsection{$RSBM$ with perturbation}

We will now prove similar results to the $SBM$ and $ER$ case but restricted to randomly generated regular graphs.  The naive spectral method of computing $\lambda_2$ is useful up to the $KS$ threshold for the problem of deciding whether a graph comes from the $RR(n,\tfrac{a+b}{2})$ or $RSBM(n,a,b)$ regimes, using the guideline that $\lambda_2(RR)\approx \sqrt{\tfrac{a+b}{2}}$ and $\lambda_2(RSBM)\approx \tfrac{a-b}{2}$ with high probability.  This means powering is of less interest for regular graphs.  However, we include the discussion in order to understand to what extent the results on $(SBM+H)^{(r)}$ are controlled by the lack of regularity in the underlying graph $SBM+H$.

First, we will give the definition of a $d$-regular random graph with $c$-clique, and a related model for a regular graph with two communities.

\begin{definition}
Let $c = o(d)$.  Then the \textit{random regular graph with $c$-clique} $RR_c(n,d)$ has vertex set $[n]$ and is constructed by first placing a random $c$-clique and then (as in the construction of a standard $d$-regular graph) adding other edges at random until the graph is $d$-regular.

Let $c=o(\min{a,b})$.  Then the \textit{regularized $SBM$ with $c$-clique} $RSBM_c(n,a,b)$ is a random graph with vertex set $n$ constructed by the following process.  First, a random $c$-clique is placed.  Then, following as in the construction of the regular stochastic block model \cite{Brito16}, divide the graph into two communities $X_1,X_2$ of equal size.  Within each community, add edges independently and uniformly at random so that the communities each induce an $a$-regular graph.  Then add edges between the communities independently and uniformly until each vertex has exactly $b$ neighbors in the other community.
\end{definition}

Similarly to the tree and $ER(n,d/n)$, we observe a threshold when $c = \sqrt{d}$ for the second eigenvalue in regular graph models with cliques added.

\begin{theorem}
With high probability, $$\lambda_2(RR_c(n,d)) = \Theta (c+\sqrt{d}).$$

Also with high probability, $$\lambda_2(RSBM_c(n,a,b)) = \Theta (c + \sqrt{\tfrac{a+b}{2}} + \tfrac{a-b}{2}).$$

\end{theorem}

\begin{proof}
The first statement follows immediately from the second, which we will prove.  Let $G$ be $RSBM_c(n,a,b)$ with the $c$-clique removed.  Weyl's inequality gives that $\lambda_2(RSBM_c)\leq c + \lambda_2(G)$, it can be seen that $\lambda_2(G) = \Theta(\sqrt{\tfrac{a+b}{2}} + \tfrac{a-b}{2})$, this gives the upper bound.

To find a lower bound on $\lambda_2(RSBM_c)$, we use the Rayleigh quotient.  Because $RSBM_c(n,a,b)$ is regular, $\lambda_2 = \max\langle Af,f\rangle /\langle f,f\rangle.$  Taking the test-function $f = 1$ on community $X_1$ and $f = -1$ on $X_2$ shows that $\lambda_2 \geq \tfrac{a-b}{2}$.  Taking $f = 1$ on the $c$-clique and $f = -1$ on $c$ other vertices shows that $\lambda_2 \geq c/2$.  Finally the Alon-Boppana result gives us $\lambda_2 \geq 2\sqrt{\tfrac{a+b}{2}}$, this completes the proof.
\end{proof}

Now we will prove a similar result to \ref{thm_ER_SBM} regarding the $RR_c$ and $RSBM_c$ random graph models.  This result also contains an implicit algorithm for determining whether a graph comes from $RSBM_c(n,a,b)$ or $RR_c(n,\tfrac{a+b}{2})$.

\begin{theorem} \label{thm_RR_RSBM}
Let $G = RSBM(n,a,b)+H$, where $|V(H)|\leq c$.  If $\sqrt{(a+b)/2} \leq (a-b)/2$ (the KS threshhold), then, independently of the choice of $H$, with high probability, $$(1-o(1))\parens{\parens{\tfrac{a-b}{2}}^r-c\sqrt{\tfrac{a+b}{2}}^{\,r-1}}\leq \lambda_2(G^{(r)})\leq (1+o(1))\parens{c\sqrt{\tfrac{a+b}{2}}^{\,r-1}+\parens{\tfrac{a-b}{2}}^r}.$$

On the other hand, if $\sqrt{(a+b)/2} \leq (a-b)/2$, then with high probability
$$(1-o(1))\parens{(r+1)\sqrt{\tfrac{a+b}{2}}-c}\sqrt{\tfrac{a+b}{2}}^{\,r-1} \leq\lambda_2(G^{(r)})\leq (1+o(1))\parens{c+(r+1)\sqrt{\tfrac{a+b}{2}}}\sqrt{\tfrac{a+b}{2}}^{\,r-1}.$$

In particular, if $G = RR(n,d)+H$ (i.e., $a = b = d$), then WHP $$(1-o(1))\parens{(r+1)\sqrt{d}-c}\sqrt{d}^{\,r-1} \leq\lambda_2(G^{(r)})\leq (1+o(1))\parens{c+(r+1)\sqrt{d}}\sqrt{d}^{\,r-1}.$$
\end{theorem}

\begin{proof} This, like Theorem \ref{thm_ER_SBM} is an application of Theorem \ref{thm_gen_clique_bound}.  The proof is similar to that of Theorem \ref{thm_ER_SBM}.

Consider the case that $G = RSBM(n,a,b)+H$, define $G' = RSBM(n,a,b)$ so that $G = G'+H$.  Clearly $D^{(i)} \leq ((a+b)/2)^i$, and so, because $$\lambda_2(G'^{(r)}) = (1+o(1))\max\parens{(r+1)\sqrt{\tfrac{a+b}{2}}^{\,r},\parens{\tfrac{a-b}{2}}^r},$$ the results follow. 
\end{proof}

We will show that in the case of $RR_c(n,d)$, the upper bound in the previous theorem is tight up to a constant factor. 

\begin{theorem}\label{thm_RR_tight}
Let $G = RR_c(n,d)$ where $c = o(d)$.  Then WHP, $$\lambda_2(G^{(r)})\geq (1-o(1))\max\parens{c , (r+1)\sqrt{d}}\sqrt{d}^{\,r-1},$$ where the $o(1)$ vanishes as $\diam(G^{(r)})$ increases.
\end{theorem}

\section{Open problems}

\begin{itemize}
    \item In Lemma \ref{lemma_girth} we show that a Ramanujan graph with girth more than $2r$ must be also $r$-Ramanujan after powering, taking advantage of the fact that all $r$-neighborhoods in that graph are trees.  Is this true in general - is every Ramanujan graph also $r$-Ramanujan (with some reasonable bound on $r$)?
    \item The converse of the previous problem - is every $r$-Ramanujan powered graph necessarily the $r$-th power of a Ramanujan graph?
    \item  Observe that we, along with our concurrent work \cite{ABRS19}, do not in general investigate the exponent in the factors  $\log(n)^{\alpha}$ which appears in our paper.  In particular in the Alon-Boppana result for powering, we see the bound $(r+1)d^{r/2}$ where $r = \ep \log(n)$, but in the bound for $\lambda_2(ER(n,d/n)^{(r)})$ we have $\log(n)^{alpha}d^r$.  Because these bounds are equivalent up to a factor of a power of $\log(n)$ we say that $ER(n,d/n)^{(r)}$ is $r$-Ramanujan.  Is it possible to better characterize the exponents of $\log(n)$ that appear in this work (especially related to the $ER$ graph) and to understand why they exist in view of the Alon-Boppana result for powering?
\end{itemize}

\section{Proofs}
\subsection{Proof of Theorem \ref{abp}}

We will prove two sub-theorems, and the result of Theorem \ref{abp} is a consequence.  The proof is based on  counting closed walks, as in the approach outlined for the Alon-Boppana Theorem in \cite{hoory}.

\begin{definition}$t_{2k}^{(r)}$ is the minimum, taken over all vertices $x\in V(G)$, of the number of closed walks of length $2k$ in $G^{(r)}$ terminating at $x$.\end{definition}

First, we will bound $\lambda_2(G^{(r)})$ in terms of $t_{2k}^{(r)}$.

\begin{theorem}\label{thm_lambda}
Let $G$ be a graph and $r\geq 1$.  Let $D$ be the diameter of $G$ and let $k$ satisfy $2k < \lceil D/r\rceil$.  Then $$\lambda_2(G^{(r)})^{2k}\geq t_{2k}^{(r)}.$$
\end{theorem}

\begin{proof}

Let $A^{(r)}$ be the adjacency matrix for $G^{(r)}$.
Let $f_1$ be an eigenfunction satisfying $A^{(r)}f_1 = \lambda_1(G^{(r)})f_1$.  By the Perron-Frobenius theorem, we can choose $f_1$ so that $f_1(z) > 0$ for all $z\in V$.

Because $A^{(r)}$ is symmetric, we can express $\lambda_2(G^{(r)})$ by the Rayleigh quotient \begin{align*}\lambda_2(G^{(r)})^{2k} = \sup_{f\perp f_1}\frac{\langle f, \parens{A^{(r)}}^{2k} f\rangle}{\langle f, f\rangle},\end{align*} where $k$ is any non-negative integer.  To obtain a lower bound on $\lambda_2(G^{(r)})$ we will set a test-function $f$ for this quotient.

For vertices $x,y\in V(G)$, set $f_{xy}(x) = f_1(y)$, $f_{xy}(y) = -f_1(x)$ and $f_{xy} \equiv 0$ otherwise.  Clearly $f_{xy}\perp f_1$.  It is well-known that $[\parens{A^{(r)}}^{2k}]_{ij}$ counts the number of walks in $G^{(r)}$ of length $2k$ that start at $i$ and end at $j$.  If $z\in V$, \begin{align*}\left(\parens{A^{(r)}}^{2k}f_{xy}\right)(z) = f_1(y)\left[\parens{A^{(r)}}^{2k}\right]_{xz}-f_1(x)\left[\parens{A^{(r)}}^{2k}\right]_{yz}.
\end{align*}

Let $D$ be the diameter of $G$, and choose $x$ and $y$ to be vertices with $\dist_G(x,y) = D$.  It follows that $d_{G^{(r)}}(x,y) = \lceil D/r \rceil$.  Choose $k$ so that $2k < \lceil D/r \rceil$.  There are no $2k$-walks in $G^{(r)}$ from $x$ to $y$ (or vice-versa), so that $\left[\parens{A^{(r)}}^{2k}\right]_{xy} = \left[\parens{A^{(r)}}^{2k}\right]_{yx} = 0$.  Now, our expression for $\lambda_2$ simplifies to \begin{align*}\lambda_2^{2k}\geq \frac{\langle f_{xy},\parens{A^{(r)}}^{2k}f_{xy}\rangle }{\langle f_{xy},f_{xy}\rangle} = \frac{f_1(y)^2\left[\parens{A^{(r)}}^{2k}\right]_{xx} + f_1(x)^2\left[\parens{A^{(r)}}^{2k}\right]_{yy}}{f_1(y)^2+f_1(x)^2} \geq t_{2k}^{(r)},\end{align*} where the last inequality holds because $t_{2k}^{(r)}\leq \left[\parens{A^{(r)}}^{2k}\right]_{zz}$ for all $z\in V$.
\end{proof}
  
Second, we will derive a lower bound on $t_{2k}^{(r)}$ in terms of the modified minimum degrees $\delta^{(i)}:0\leq i\leq r$.

\begin{theorem}\label{thm_tree}
Let $r$ and $k$ be positive integers. \begin{align*}\parens{t^{(r)}_{2k}}^{1/(2k)}\geq \parens{1-o(1)} \sum_{i = 0}^{r}\sqrt{\delta^{(i)}\delta^{(r-i)}}. \end{align*}
\end{theorem}  Here, the $o$ notation is in terms of $k$, treating $d$ and $r$ as constants.  

Fix $x$ to be vertex with the minimum number of closed walks of length $2k$ in $G^{(r)}$ terminating at $x$.  By definition, that number is $t_{2k}^{(r)}$.

For exposition, we will start by outlining a proof of this result for simple graphs, first given in \cite{hoory}.  

\begin{theorem}

Let $G$ be a $d$-regular graph and $x\in V(G)$.  Consider $t_{2k}$, the number of closed walks on $G$ of length $2k$ terminating at $x$.  Then $$\parens{t_{2k}}^{1/(2k)}\geq \parens{1-o(1)}2\sqrt{d-1}.$$

\end{theorem}

The following is the proof outline:

\begin{itemize}
\item We wish to obtain a lower bound on the number of closed walks of length $2k$ from $x$ to $x$ in $G$.
\item Consider only those walks that are ``tree-like", i.e., walks with the following description.   We start a walk at the root $x$.  We construct a tree labelled with vertices of $G$: if at some step we are at $y$, we can either move to the parent of $y$ or generate a child of $y$ by moving to some other neighbor.  We want to count only those walks that trace all edges of the tree twice and so terminate at the original root $x$; for example:
\begin{itemize}
\item $x,y,z,y,x$ is tree-like.
\item $x,y,z,y,z,y,x$ is also tree-like : in this case $z$ is the label for two children of the same vertex.
\item $x,y,z,x$ is closed but not tree-like, because it ends on a node of depth $3$ rather than the root.
\end{itemize}
(Another characterization is that tree-like walks correspond to closed walks on the cover graph of $G$.)
\item In a tree-like walk of length $2k$, the sequence of moving to a child / moving to a parent makes a Dyck word of length $2k$, of which there are the $k$-th Catalan number $C_k$.
\item For a Dyck word, there are at least $d-1$ choices of which child to generate at each of those $k$ steps, and exactly $1$ choice at each step that we return to the parent node.  Note here that $d-1 = \delta^{(1)}$ for a $d$-regular graph.
\item In total we find that $$t_{2k}\geq (d-1)^kC_k\gtrapprox \frac{\parens{2\sqrt{d-1}}^{2k}}{k^{3/2}},$$ the bound follows.
\end{itemize}

We will to use a similar process to bound $t_{2k}^{(r)}$.  A simple suggestion is to consider walks in $G^{(r)}$ that correspond to closed walks on the cover graph of $G^{(r)}$.  The argument outlined above works (as it will for the cover graph of \textit{any} graph) and we see that $$t_{2k}^{(r)}\geq (d^{(r)}-1)C_k\approx \parens{2\sqrt{d^{(r)}-1}}^{2k},$$ where $d^{(r)}$ is the minimum degree in $G^{(r)}$.

The problem is that this is not tight: if we approximate $\delta^{(i)} \approx d^{(i)} \approx d^i$ (as is the case in a $d$-regular graph with girth larger than $2i$), this result gives $$\parens{t_{2k}^{(r)}}^{1/(2k)} \gtrapprox 2d^{r/2},$$  while our theorem has the result
$$\parens{t_{2k}^{(r)}}^{1/(2k)} \gtrapprox (r+1)d^{r/2}.$$  In order to prove the theorem we must improve this method.  Consider a class of walks we are not counting:

Let $r = 2$, suppose $x,y,z$ are vertices with a common neighbor $w$.  $x,y,z,x$ is a closed walk in $G^{(2)}$ but it does not correspond to a closed walk in the cover graph of $G^{(2)}$.  However, the underlying $G$-walk $x,w,y,w,z,w,x$ \textit{does} correspond to a closed walk on the cover graph of $G$.  This is the type of walk that we want to count, but failed to do so in our first attempt.

So instead we can try to count the number of walks in the cover graph of $G$ that are the underlying walk for some $G^{(r)}$-walk.  But this suggestion introduces a new issue: a given $G^{(r)}$-walk may correspond to several different underlying walks on the cover graph of $G$.  For instance (again in the setting $r=2$), if $x\sim y\sim z\sim w\sim x$, then the $G^{(2)}$-walk $x,z,x$ corresponds to both $x,y,z,y,x$ and $x,w,z,w,x$.  If we want to count the underlying walks as a lower bound on the number of $G^{(2)}$-walks, we must disregard all but at most $1$ of the underlying walks that correspond to a given $G^{(2)}$-walk.  For this reason we introduce a system of canonical paths between neighbors in $G^{(r)}$ that we follow when generating an underlying walk on $G$.

\begin{definition}[Canonical $i$-path]For all ordered pairs of distinct vertices $(v,w)$ with $\dist_G(v,w) = i$, we arbitrarily choose a canonical path of length $i$ from $v$ to $w$.\end{definition}

We now define the set of tree-like walks in $G^{(r)}$, which we will later count.  The tree-like walks are analogous to walks on the cover graph of $G$.  We first define the related concept of a sequence of canonically constructed walks.

\begin{definition}[Sequence of $r$-canonically constructed walks]
Let $k>0$.  A sequence of $r$-canonically constructed walks is a sequence $W_0,\dots,W_{2k}$ of walks that is constructed according to the following process:

Assume $W_i$ has length at least $r$, so it can be expressed as $W_i = \dots, v_0,v_1,\dots, v_r$.  We can then make a move of type $m$ to obtain $W_{i+1}$, where $m$ is an integer satisfying $0\leq m\leq r$.

In a move of type $m$, we start by removing the last $r-m$ vertices from $W_i$, leaving $\dots v_0,v_1,\dots, v_m$.  Let $y$ be a vertex with $\dist_G(v_m,y) = m$ and $\dist_G(v_{m-1},y)\geq m$.  We can append to our sequence the canonical walk from $v_m$ to $y$ of length $m$ to obtain $W_{i+1}$.  Observe that $\dist_G(v_r,y)\leq \dist_G(v_r,v_m) + \dist_G(v_m,y) = (r-m)+m = r$, so that $v_r$ and $y$ are neighbors in $G^{(r)}$.

We require $W_0 = W_{2k} = x$.  To find $W_1$, we are allowed to make a move of type $r$ (removing no vertices from $W_0$).  When $0 < i < 2k$, we require that $W_i$ has length at least $r$.
\end{definition}

The motivation for this definition is that we are taking a $G^{(r)}$-walk on the endpoints of $W_0,\dots,W_{2k}$: observe that in such a sequence of walks, the endpoints of $W_i$ and $W_{i+1}$ have graph distance at most $r$.

Note that the final move from $W_{2k-1}$ to $W_{2k}$ must be of type $0$, where $W_{2k-1}$ has length $r$ and $W_{2k}$ has length $0$.

Note that if there is a move of type $m$ from $W_i$ to $W_{i+1}$ where $W_{i+1}$ ends in $y$, no move of type $<m$ from $W_i$ can end in $y$.  Suppose on the contrary there is $j<m$ so that some move of type $j$ ends in $y$.  Then $\dist_G(v_{m-1},y)\leq \dist_G(v_{m-1},v_j) + \dist_G(v_j,y) \leq (m-1-j)+j = m-1$.  This contradicts that $\dist_G(v_{m-1},y) \geq m$.

Also note that, given $W_i$ is long enough that moves of type $m$ are allowed, the number of possible moves of type $m$ that can generate distinct possibilities for $W_{i+1}$ is at least $\delta^{(m)}$.

\begin{definition}[Tree-like walk in $G^{(r)}$]
Let $k>0$.  A tree-like walk of length $2k$ in $G^{(r)}$ is the closed walk made by the endpoints of an sequence of $r$-canonically constructed walks starting at $x$ and having length $2k$.
\end{definition}

We will count the number of tree-like walks in $G^{(r)}$.  First, it is useful to prove an equivalence between the tree-like walk $x_0,\dots, x_{2k}$ and underlying sequence of $r$-canonically constructed walks $W_0,\dots, W_{2k}$.

\begin{theorem}\label{thm_treeseq}
Let $k > 0$.  Let $\mathcal{W}$ be the set of sequences of $r$-canonically constructed walks.  Let $\mathcal{X}$ be the set of tree-like walks in $G^{(r)}$.  Then the function that takes an sequence of $r$-canonically constructed walks and outputs the tree-like walk consisting of its final vertices is a bijection between $\mathcal{W}$ and $\mathcal{X}$.
\end{theorem}

\begin{proof}
It is clear from the definition of tree-like walks that this function is surjective.  It remains to show that it is injective.  We will argue that it is possible to recover the sequence of $r$-canonically constructed walks $W_0,\dots, W_{2k}$ given the tree-like walk $x_0,\dots, x_{2k}$.

Start with $W_0 = x = x_0$.  Given $x_0,\dots x_k$ and $W_i$ we wish to find $W_{i+1}$.  We first determine the type of the move from $W_i$ to $W_{i+1}$:

Write $W_i = \dots, v_0,v_1,\dots v_r$.  Let $m$ be the least integer so that $\dist_G(x_{i+1},v_m) \leq m$.  Then the move from $W_i$ to $W_{i+1}$ must be of type $m$, for the following reasons:  assume for contradiction that the type is not $m$.

If the move is in fact of type $j < m$, then $\dist_G(x_{i+1},v_j) = j$, and so $m$ is not the least integer satisfying $\dist_G(x_{i+1},v_m) \leq m$, this is a contradiction.

Instead the move may be of type $j > m$.  But then, $\dist_G(v_{j-1}, x_{i+1})\leq \dist_G(v_{j-1},v_m) + \dist_G(v_m, x_{i+1})\leq (j-1-m)+m = j-1$.  This contradicts the assumption that $\dist_G(v_{j-1},x_{i+1})\geq j$ in a move of type $j$.

So the move must be of type $m$.  We can determine $W_{i+1}$ by removing all vertices after $v_m$ from $W_i$ and then appending the canonical walk from $v_m$ to $x_{i+1}$.
\end{proof}

Because every tree-like walk is a walk, $t_{2k}^{(r)}\geq \abs{\mathcal{X}}$.  It follows from Theorem \ref{thm_treeseq} that $t_{2k}^{(r)}\geq \abs{\mathcal{W}}$.  In order to prove Theorem \ref{thm_tree}, we will look for a lower bound on $|\mathcal{W}|$.

First, given $W_0,\dots, W_i$, we count the number of ways that we can make a move of type $m$ to generate a walk $W_{i+1}$.  Assuming $W_i = \dots, v_0,v_1,\dots v_r$, the last vertex of $W_{i+1}$ can be any vertex $y$ so that $\dist_G(v_m,y) = m$ and $\dist_G(v_{m-1},y) \geq m$.  Because $(v_m, v_{m-1})\in E$, the number of such vertices is at least $\delta^{(m)}$.

Now, suppose $m_1,\dots, m_{2k}$ is a legal sequence of move types that generate a closed walk. 
Here, $m_i$ represents the move type between $W_{i-1}$ and $W_i$.  The number of tree-like walks with such a sequence is at least $$\prod_{i=1}^{2k} \delta^{(m_i)}.$$

The remaining difficulty is to describe the allowed sequences $m_1,\dots, m_{2k}$ of move types that will result in $r$-canonically constructed sequences of walks.  Here, it is convenient to introduce the sequence of length changes $p_1,\dots, p_{2k}$:

\begin{definition}
Let $W_0,\dots, W_{2k}$ be a sequence of $r$-canonically constructed walks.  Let $\text{len}(W_i)$ represent the length (as a walk) of $W_i$.  Then the sequence of length changes is $p_1,\dots, p_{2k}$, where $p_i = \text{len}(W_{i})-\text{len}(W_{i-1})$.
\end{definition}

Because $W_{i}$ is obtained by removing $r-m_i$ vertices from $W_{i-1}$ and then appending $m_i$ vertices to the result, $p_i = 2m_i-r$.  This gives an obvious equivalence between the allowed sequences of move types and the allowed sequences of length changes.  So we will instead consider the problem of which sequences of length changes result in an $r$-canonically constructed sequence of walks.

Note that the number of tree-like walks with length change sequence $p_1,\dots, p_{2k}$ is at least $$\prod_{i=1}^{2k}\delta^{(\tfrac{r+p_i}{2})}.$$

It is straightforward to see that $p_1 = r$ and $p_{2k} = -r$ for any allowed sequence.  Because $W_0 = W_{2k} = x$, the total length change is $$0 = \sum_{j=1}^{2k}p_j.$$  The other requirement is the $\text{len}(W_i)\geq r$ whenever $1\leq i\leq 2k-1$; i.e., $$0\leq \sum_{j=2}^{i}p_j,$$ whenever $1\leq i\leq 2k-1$ (observing that $r = p_1 = \text{len}(W_1)$).

Notice that $p_i\in \{-r,-r+2,-r+4,\dots, r-4,r-2,r\}$.
Given that $p_1 = r,p_{2k} = -r$, we have the problem of finding length change sequences $p_2,\dots,p_{2k-1}$ whose sum is $0$ and the partial sum of the first $i$ terms is non-negative for any value of $i$.  For $r=1$, $p_i\in \{-1,1\}$, these sequences are equivalent to the Dyck words of length $2k-2$, which are well-known to be enumerated by the Catalan number $C_{k-1}$.  For $r=2$, $p_i\in \{-2,0,2\}$, these are equivalent to the Motzkin paths of length $2k-2$, which were first studied in \cite{Mot48}.  For general $r$ the problem of enumerating such sequences has recently been posed in \cite{EKM18}.  There is currently no known closed-form enumeration or non-trivial relation to another problem in the case $r\geq 3$.  For that reason we will restrict ourselves to a subset of the length change sequences; this will make the remaining computations more straightforward, and our final bound on $\lambda_2(G^{(r)})$ will still be tight.  We will only consider length change sequences $p_1,\dots, p_{2k}$ satisfying the conditions:

\begin{itemize}
\item $p_1 = r$ and $p_{2k} = -r$.
\item If $j\in \{-r,-r+2,\dots, r-2,r\}$ and $j>0$ the subsequence that consists of only entries $j$ and $-j$ is a Dyck word starting with type $j$.
\item The subsequence of entries $r$ and $-r$ is still a Dyck word if the entries $p_1 = r$ and $p_{2k} = -r$ are removed.
\item If $r$ is even, there may also be entries $p_i = 0$, corresponding to moves of type $r/2$.
\end{itemize}

In any tree-like walk with such a sequence of length changes, the Dyck word criteria require that $\text{len}(W_i) = r+\sum_{j=2}^i p_j\geq r$ whenever $1 \leq i \leq {2k-1}$.  In addition, because the number of length changes of $j$ and $-j$ are always equal, $\text{len}W_{2k}=0$.

In order to count over the sequences of $r$-canonically constructed walks that have such a length change sequence, we need to prove a lemma.  The well-known binomial identity reveals that \begin{align*}
\sum_{i_1 + \dots i_k = n}\binom{n}{i_1,\dots, i_k}x_1^{i_1}\dots x_k^{i_k} = \parens{\sum_{i=1}^k x_i}^n.\end{align*}  In our results, we will use similar sums, except, instead of summing over all partitions of $n$, we sum over only even partitions of $n$.  (We require that $n$ be even so that such a partition is possible.)  In this lemma we bound the sum over even partitions in terms of the binomial identity.

\begin{lemma} \label{lemma_binom}
Suppose $x_1,\dots,x_k \geq 0$ and $2n$ is a non-negative even integer.  Then

\begin{align*}\sum_{2m_1+\dots 2m_k = 2n} \binom{2n}{2m_1,\dots, 2m_k}x_1^{2m_1}\dots x_k^{2m_k}\geq \frac{1}{2^{k-1}} \parens{\sum_{i=1}^k x_i}^{2n},\end{align*}

where the first sum is over all $k$-tuples of non-negative even integers that sum to $2n$.

\end{lemma}

\begin{proof}

Consider vectors $j\in \{-1,1\}^k$.

\begin{align*}\sum_j \parens{\sum_{i=1}^k j_ix_i}^{2n} = 2^k\sum_{2m_1+\dots 2m_k = 2n} \binom{2n}{2m_1,\dots, 2m_k}x_1^{2m_1}\dots x_k^{2m_k}.\end{align*} Here, each term in the right-hand sum is counted once for each of the $2^k$ choices of $j$.  Any term of the binomial expansion that does not correspond to an even partition of $2n$ is positive for exactly $2^{k-1}$ values of $j$ and negative for the other $2^{k-1}$.  In the sum over $j$ such a term cancels.

Also, \begin{align*}\sum_j \parens{\sum_{i=1}^k j_ix_i}^{2n} \geq 2\parens{\sum_{i=1}^k x_i}^{2n},\end{align*} because the left-hand side is a sum of non-negative quantities that contains the right hand side twice: when $j \equiv 1$ and $j\equiv -1$.  Combining we find \begin{align*}2^k\sum_{2m_1+\dots 2m_k = 2n} \binom{2n}{2m_1,\dots, 2m_k}x_1^{2m_1}\dots x_k^{2m_k} \geq 2\parens{\sum_{i=1}^k x_i}^{2n},\end{align*} the result immediately follows.

\end{proof}

Finally we are able to give a lower bound on $t_{2k}^{(r)}$.  We will prove the bound separately in the cases that $r$ is odd and $r$ is even: the difference is that when $r$ is even we must consider length change terms $p_i = 0$, corresponding to moves of type exactly $r/2$.

\begin{proof}{Proof of Theorem \ref{thm_tree}}

First, suppose $r$ is odd.  In this case, the number of accepted length change sequences is \begin{align*}
\sum_{n_1+n_3+\dots +n_{r} = k-1}\binom{2k-2}{2n_1,2n_3\dots, 2n_{r}}\prod_{j=1\atop j\text{ odd}}^{r}C_{n_j},\end{align*} where $n_j$ is the number of times types $j$ (and $-j$) appear in the length change sequence.  The total number of tree-like walks on $G^{(r)}$ that correspond to such sequences is
\begin{align*}
t_{2k}^{(r)}&\geq \sum_{n_1+n_3+\dots +n_{r} = k-1}\binom{2k-2}{2n_1,2n_3\dots, 2n_{r}}\prod_{j=1\atop j\text{ odd}}^{r}C_{n_j}\parens{\delta^{(\tfrac{r+j}{2})}}^{n_j}\parens{\delta^{(\tfrac{r-j}{2})}}^{n_j}\\
&=\parens{1+O\parens{\tfrac{1}{k}}} \sum_{n_1+n_3+\dots +n_{r} = k-1}\binom{2k-2}{2n_1,2n_3\dots, 2n_{r}}\prod_{j=1\atop j\text{ odd}}^{r}\frac{\parens{4\delta^{(\tfrac{r+j}{2})}\delta^{(\tfrac{r-j}{2})}}^{n_j}}{n_j^{3/2}\sqrt{\pi}}\\
& \geq \parens{1+O\parens{\tfrac{1}{k}}}\parens{\frac{r^3}{8k^3\sqrt{\pi}}}^{r/2}\sum_{n_1+n_3+\dots +n_{r} = k-1}\binom{2k-2}{2n_1,2n_3\dots, 2n_{r}}\parens{2\sqrt{\delta^{(\tfrac{r+j}{2})}\delta^{(\tfrac{r-j}{2})}}}^{2n_j}\\
& \geq \parens{1+O\parens{\tfrac{1}{k}}}\parens{\frac{r^3}{16k^3\sqrt{\pi}}}^{r/2}\parens{\sum_{j = 1\atop j\text{ odd}}^{r}2\sqrt{\delta^{(\tfrac{r+j}{2})}\delta^{(\tfrac{r-j}{2})}}}^{2k-2}\\
& = \parens{1+O\parens{\tfrac{1}{k}}}\parens{\frac{r^3}{16k^3\sqrt{\pi}}}^{r/2}\parens{\sum_{i = 0}^{r}\sqrt{\delta^{(i)}\delta^{(r-i)}}}^{2k-2}.
\end{align*}

Here, the first equality is a standard approximation for $C_n$ and the last inequality is the first statement of Lemma \ref{lemma_binom}.  The last equality is a result of re-indexing from length change $j$ to move type $i = \tfrac{r+j}{2}$.  The result of Theorem \ref{thm_tree} follows (though it remains to prove the theorem when $r$ is even): \begin{align*}\parens{t_{2k}^{(r)}}^{1/(2k)}\geq \parens{1-o(1)}\sum_{i = 0}^{r}\sqrt{\delta^{(i)}\delta^{(r-i)}}.\end{align*}

If instead $r$ is even, the computations are complicated by the existence of moves of type $r/2$.  The number of accepted move type sequences is \begin{align*}
\sum_{n_2 + n_4 + \dots +n_{r} = k-1-n_0/2}\binom{2k-2}{n_0,2n_2,2n_4,\dots  2n_{r}}\prod_{j=2\atop j\text{ even}}^{r}C_{n_j},\end{align*} where $n_j$ is the number of times $j$ (and $-j$) appear in the length change sequence.  Note that our sum condition requires that $n_0$ is even; in principle we could allow $n_0$ and the walk length $2k$ to both be odd (so, the walk length would be $2k+1$), which does not make a major difference to the computation or result.  Using almost the same argument as before, the number of tree-like walks on $G^{(r)}$ that correspond to such sequences is
\begin{align*}
t_{2k}^{(r)}&\geq \sum_{n_2 + n_4 + \dots +n_{r} \atop = k-1-n_0/2}\binom{2k-2}{n_0,2n_2,2n_4,\dots  2n_{r}}\parens{\prod_{j=2\atop j\text{ even}}C_{n_j}\parens{\delta^{(\tfrac{r+j}{2})}}^{n_j}\parens{\delta^{(\tfrac{r-j}{2})}}^{n_j}}\parens{\delta^{(r/2)}}^{n_{r/2}}\\
&=\parens{1+O\parens{\tfrac{1}{k}}} \sum_{n_2 + n_4 + \dots +n_{r} \atop = k-1-n_0/2}\binom{2k-2}{n_0,2n_2,2n_4,\dots  2n_{r}}\parens{\prod_{j=2\atop j\text{ even}}\frac{\parens{4\delta^{(\tfrac{r+j}{2})}\delta^{(\tfrac{r-j}{2})}}^{n_j}}{n_j^{3/2}\sqrt{\pi}}}\parens{\delta^{(r/2)}}^{n_{r/2}}\\
& \geq \parens{1+O\parens{\tfrac{1}{k}}}\parens{\frac{r^3}{8k^3\sqrt{\pi}}}^{r/2}\sum_{n_2 + n_4 + \dots +n_{r} \atop = k-1-n_0/2}\binom{2k-2}{n_0,2n_2,2n_4,\dots  2n_{r}}\\ &\parens{\prod_{j=2\atop j\text{ even}}\parens{2\sqrt{\delta^{(\tfrac{r+j}{2})}\delta^{(\tfrac{r-j}{2})}}}^{2n_j}}\parens{\delta^{(r/2)}}^{n_{r/2}}\\
& \geq \parens{1+O\parens{\tfrac{1}{k}}}\parens{\frac{r^3}{16k^3\sqrt{\pi}}}^{r/2}\parens{\delta^{(r/2)}+\sum_{j = 2\atop j\text{ odd}}^{r}2\sqrt{\delta^{(\tfrac{r+j}{2})}\delta^{(\tfrac{r-j}{2})}}}^{2k-2}\\
& = \parens{1+O\parens{\tfrac{1}{k}}}\parens{\frac{r^3}{16k^3\sqrt{\pi}}}^{r/2}\parens{\sum_{i = 0}^{r}\sqrt{\delta^{(i)}\delta^{(r-i)}}}^{2k-2}.
\end{align*}

As in the case where $r$ is odd, the method is to use the standard approximation on $C_n$ and then to cite Lemma \ref{lemma_binom} in order to express the sum as a binomial expansion.  The result of Theorem \ref{thm_tree} follows: \begin{align*}\parens{t_{2k}^{(r)}}^{1/(2k)}\geq \parens{1-o(1)}\sum_{i = 0}^{r}\sqrt{\delta^{(i)}\delta^{(r-i)}}.\end{align*}

\end{proof}

Now, combining Theorems \ref{thm_lambda} and \ref{thm_tree}, we see that \begin{align*} \lambda_2(G^{(r)})\geq \parens{1-o(1)}\sum_{i = 0}^{r}\sqrt{\delta^{(i)}\delta^{(r-i)}} = \parens{1-o(1)}(r+1)\hat{d}_r\,^{r/2}(G),\end{align*} where the $o$ notation is in terms of $k = \tfrac{1}{2}\lceil \text{diam}(G)/r \rceil$.  Under the hypotheses of Theorem \ref{abp}, we have $\text{diam}(G_n) = \omega(1)$ and $r_n = \epsilon\cdot\text{diam}(G_n)$, so that $k_n = \tfrac{1}{2}\lceil 1/\epsilon \rceil$.  Letting $\epsilon \to 0$, we have $k_n = \omega(1)$.  The result of Theorem \ref{abp} follows.

\subsection{Proof of Lemma \ref{lemma_girth}}
\begin{proof}

Per the hypotheses, let $G$ be a $d$-regular Ramanujan graph with girth $g$.  Assume that $2r < g$.

It is straightforward to compute that $\delta^{(i)} = (d-1)^i$ for all values $0\leq i \leq r$.  Our results show that  $\lambda_2(A^{(r)})\geq (1-o(1))(r+1)(d-1)^{r/2}$.  We will show that this is tight up to a factor $1+o_d(1)$.  To do so, we will find an upper bound for $\lambda_2(A^{(r)})$.
Because $G$ is isomorphic to a $d$-regular tree in any $r$-neighborhood, there is a recursive formula for $A^{(r)}$:

$$A^{(r)} = AA^{(r-1)}-(d-1)A^{(r-2)}$$ if $2\leq r$.  The base cases are $A^{(0)} = I$ and $A^{(1)} = A+I$.

We briefly justify this recursion:  $[AA^{(r-1)}]_{ij}$ counts the number of neighbors $v$ of $j$ satisfying $\dist_G(i,v)\leq r-1$.  Because of the girth bound, this number is $1$ if $\dist_G(i,j) = r$ or $r-1$.  It is $0$ if $\dist_G(i,j) > r$ and $d$ if $\dist_G(i,j) \leq r-2$.  Because $[A^{(r)}]_{ij} = 1$ iff $\dist_G(i,j)\leq r$, we subtract $(d-1)A^{(r-2)}$ from the previous term.

It is easy to see that there is a sequence of polynomials $p^{(r)}$ so that $A^{(r)} = p^{(r)}(A)$ (though it requires some effort to compute $p^{(r)}$.)  If $\lambda_1\geq \lambda_2\geq\dots$ are the eigenvalues of $A$, then $p^{(r)}(\lambda_1), p^{(r)}(\lambda_2),\dots$ are the eigenvalues of $A^{(r)}$ (but note that these values are not necessarily ordered.)  The largest eigenvalue of $A^{(r)}$ is $p^{(r)}(d)$, achieved by eigenvector $v\equiv 1$.  Because all other eigenvalues of $A^{(r)}$ are $p^{(r)}(\lambda)$ for some $|\lambda |< 2\sqrt{d-1}$, $$\lambda_2(A^{(r)}) \leq \max_{|x|< 2\sqrt{d-1}} |p^{(r)}(x)|.$$

We must now compute $p^{(r)}$.  Given the recursive formula $p^{(r)}(x) = xp^{(r-1)}(x) - (d-1)p^{(r-2)}(x)$ with base cases $p^{(0)}(x) = 1, p^{(1)}(x) = x+1$, it is straightforward to write the solution:  if $\abs{x}< 2\sqrt{d-1}$, \begin{align*}p^{(r)}(x) &= \parens{\frac{1}{2} - \frac{i}{2}\frac{x/2 + 1}{\sqrt{(d-1)-x^2/4}}}\parens{\frac{x}{2} + i\sqrt{(d-1)-x^2/4}}^r\\ &+\parens{\frac{1}{2} + \frac{i}{2}\frac{x/2 + 1}{\sqrt{(d-1)-x^2/4}}}\parens{\frac{x}{2} - i\sqrt{(d-1)-x^2/4}}^r\end{align*}

Define $\theta = arccos\parens{\frac{x/2}{\sqrt{d-1}}}$.

\begin{align*}
p^{(r)}(x) &= \frac{1}{2}\parens{1-i\frac{1}{\sqrt{d-1}\sin\theta}-i\frac{\cos\theta}{\sin\theta}}\parens{d-1}^{r/2}\parens{\cos\theta + i\sin\theta}^r\\ &+ \frac{1}{2}\parens{1+i\frac{1}{\sqrt{d-1}\sin\theta}+i\frac{\cos\theta}{\sin\theta}}\parens{d-1}^{r/2}\parens{\cos\theta - i\sin\theta}^r.\\
&= \frac{(d-1)^{r/2}}{2}\parens{1-i\frac{1}{\sqrt{d-1}\sin\theta}-i\frac{\cos\theta}{\sin\theta}}\parens{\cos(r\theta)+ i\sin(r\theta)}\\&+\frac{(d-1)^{r/2}}{2}\parens{1-i\frac{1}{\sqrt{d-1}\sin(-\theta)}-i\frac{\cos(-\theta)}{\sin(-\theta)}}\parens{\cos(-r\theta)+ i\sin(-r\theta)}\\
& = \parens{d-1}^{r/2}\parens{\cos(r\theta)+\frac{\sin(r\theta)}{\sqrt{d-1}\sin\theta}+ \frac{\cos\theta\sin(r\theta)}{\sin\theta}}.
\end{align*}

It follows that \begin{align*}\abs{p^{(r)}(x)}\leq \parens{d-1}^{r/2}\parens{1+ \frac{r}{\sqrt{d-1}} + r} = \parens{1+o_d(1)}(r+1)(d-1)^{r/2}.\end{align*}  Holding $r$ as a constant, our lower bound $\lambda_2(A^{(r)})= (1+o_d(1)) (r+1)d^{r/2}$ is tight.  This completes the proof.
\end{proof}

\subsection{Proof of Theorem \ref{thm_tr_cliques}}

\begin{proof}
First, we prove the upper bound on $\rho$.  Express $E(G) = E_1 \cup E_2$, where $E_2$ consists of edges corresponding to an $\leq r$-path passing through the $c$-clique of $T_{d,c}$ and $E_1$ consists of all other edges of $G$.  Write $G_1 = (V(G),E_1)$ and $G_2 = (V(G),E_2)$.

As before, $\max(\rho(G_1)+\rho(G_2) \leq \rho(G)\leq \rho(G_1)+\rho(G_2).$   Because of Theorem \ref{abp}, we have $\rho(G_1) = (1+o(1))(r+1)\sqrt{d}$.  It remains to compute $\rho(G_2)$.

If $x\sim_{G_2} y$, it must be the case that there are vertices $v,w$ within the copy of $K_c$ so that $\dist_G(x,v) + \dist_G(y,w) \leq r-1$, i.e., there is a path of length at most $r$ between $x$ and $y$, and that path passes through the copy of $K_c$.  Partition $E(G_2)$ into $F_q:0\leq q\leq r-1$, where $\{x,y\}\in F_q$ if $\{x,y\}\in E(G_2)$ and $\dist_T(x,K_c) + \dist_G(y,K_c) = q.$  Let $B_q$ be the adjacency matrix of $F_q$, clearly $\norm{A_2}\leq \sum_q \norm{B_q}$.

Now we will bound $\norm{B_q}$.  Suppose $\dist_T(x,K_c) = i$, then $x$ has at most $c\cdot d^{q-i}$ $E_q$-neighbors $y$, and no such neighbors if $i > q$.

Define $f_q$ so that $f_q(x) = \sqrt{d}^{\,q-i}$ if $\dist_T(x,K_c) = i$.  For such $x$,  $$\frac{Af(x)}{f(x)} \leq \frac{c\cdot d^{q-i}\sqrt{d}^i}{\sqrt{d}^{\,q-i}}= c\sqrt{d}^{\,q}.$$  It follows that $\norm{B_q}\leq c \sqrt{d}^{\,q}$, and thus $\norm{A_2}\leq \sum_{q = 0}^{r-1}c\sqrt{d}^{\,q} = (1+o(1))c\sqrt{d}^{\,r-1}$, this completes the proof of the upper bound.

We will now prove the lower bound on $\rho$.  Recall that $$\rho(G) \geq \parens{\frac{\langle A^kf,f\rangle}{\langle f,f\rangle}}^{1/k}$$ for any positive integer $k$ and $f$ satisfying $0 < \langle f,f\rangle <\infty$.  Let $2k$ be an even positive integer and set $f = 1_x$ for some $x$ that is a vertex of the $K_c$, so that $\rho(G)\geq \parens{t_{2k}(x)}^{1/2k}$.  We will exhibit two types of $k$-walks in $G$ terminating at $x$ in order to give a lower bound on $t_{2k}(x)$.

First, consider walks on $G$ that do not include any path that uses an edge of $r$.  This is equivalent to evaluating $t_{2k}$ on $T_d^{(r)}$, we have found in the proof of Theorem \ref{abp} that $(t_{2k})^{1/2k} \geq (1-o(1))(r+1)\sqrt{d}^{\,r}$.

Next, consider walks in which every odd-numbered step is an $r$-path that starts with a move in $K_c$ and then contains $r-1$ moves along the copy of $T_d$ attached to that vertex.  Every even-numbered step uses $r-1$ moves returning to the $K_c$ and then ends with a move in $K_c$.  The number of $2k$-walks of this type is $(c^2d^{r-1})^k$, and so $t_{2k}(x)^{1/2k}\geq \sqrt{c^2 d^{r-1}} = c \sqrt{d}^{\,r-1}$.  Combining these two counts of walks gives the lower bound.

\end{proof}

\subsection{Proof of Theorem \ref{thm_RR_tight}}

\begin{proof}
Recall from the proof of Theorem \ref{abp} that $\lambda_2(G^{(r)})^k$ is bounded below by $$\max_{f\perp f_1} \frac{\langle \parens{A^{(r)}}^kf,f\rangle}{\langle f,f\rangle},$$ where $k$ is any positive integer and where $f_1$ is the eigenfunction corresponding to $\lambda_1(G)$.  As $G$ is regular, $f_1 = \vec{1}$ with eigenvalue $\lambda_1 = d$.  As in the proof of the Alon-Boppana theorem, we will set a test-function $f = 1_x-1_y$ for two vertices $x$ and $y$.  In particular, we will choose $x,y$ to maximize $\dist_G(x,y)$ over all pairs for which $x$ is a vertex of the $c$-clique.  Observe that $\dist_G(x,y)\geq \tfrac{1}{2}\text{diam}(G)$  Choose $k$ even so that $kr< \dist_G(x,y) \leq(k+2)r$; as $n$ increases and $d$ remains constant, $k\to\infty$.  Then, after evaluating for $f$, $$\lambda_2(G^{(r)})^k \geq \tfrac{1}{2}\parens{[(A^{(r)})^k]_{xx}+[(A^{(r)})^k]_{yy}}.$$  It remains to find a lower bound on $[(A^{(r)})^k]_{xx}$ and $[(A^{(r)})^k]_{yy}$, i.e., the number of closed $k$-walks in $G^{(r)}$ terminating at $x$ and $y$ respectively.

For any vertex $v$ of $G^{(r)}$, the number of closed $k$-walks is WHP at least $$\parens{(1-o_k(1))(r+1)\sqrt{(d-c)d^{r-1}}}^k,$$ counting only those walks that don't use any edge of the $c$-clique.  This is an application of the  Alon-Boppana bound calculated in Theorem \ref{abp}; following the method of Lemma \ref{lemma_reg} after removing the edges of the $c$-clique, $WHP$ every vertex in $G^{(i)}:i\leq r$ has minimum degree $(1-o_k(1))(d-c)d^{i-1}$.  

Then, as in the proof of Theorem \ref{thm_tr_cliques}, we can also count a class of closed walks in $G^{(r)}$ ending at $v$ for which every step passes through the $c$-clique.  In such a walk, every odd-numbered step consists of an $G^{(r)}$-edge that corresponds to the sequence of a walk in $H$ followed by $r-1$ moves along the remaining edges, and the even-numbered step returns to $H$ and ends on any vertex of $H$.  The number of such walks is $\parens{(1-o_k(1))c^2(d-c)d^{r-2}}^{k/2} = \parens{(1-o_k(1))c\sqrt{d-c}\sqrt{d}^{\,r-2}}^k$, using our previous approximation of the degree in a random regular graph after powering.  

Using these bounds, we see that $$\lambda_2^k \geq \tfrac{1}{2}\parens{\parens{(1-o_k(1))c\sqrt{d-c}\sqrt{d}^{\,r-2}}^k+2\parens{(1-o_k(1))(r+1)\sqrt{(d-c)d^{r-1}}}^k},$$
it follows that $$\lambda_2\geq (1-o_{\diam(G^{(r)})}(1))\max(c,(r+1)\sqrt{d})\sqrt{d-c}\sqrt{d}^{\,r-2}.$$  The result is found by using the assumption $c = o(d)$.

\end{proof}

\bibliography{gen_sbm}
\bibliographystyle{plain}

\end{document}